\documentclass[preprint,12pt]{elsarticle}
\usepackage{amsmath,amsthm,amssymb}
\usepackage{epstopdf}
\usepackage{makecell}
\usepackage{booktabs}
\usepackage{multirow}
\usepackage{diagbox}
\usepackage{extarrows}
\biboptions{numbers,sort&compress}
\usepackage{ulem}
\usepackage[colorlinks,linkcolor=blue,anchorcolor=blue,citecolor=blue]{hyperref}
\usepackage{cleveref}
\usepackage{bm,bbm}

\newtheorem{lemma}{Lemma}
\newtheorem{theorem}{Theorem}
\newtheorem{proposition}{Proposition}
\newtheorem{assumption}{Assumption}

\newtheorem{example}{Example}
\numberwithin{theorem}{section} 
\numberwithin{equation}{section}
\numberwithin{lemma}{section}
\numberwithin{example}{section}
\numberwithin{definition}{section}

\allowdisplaybreaks[4]

\journal{Journal}

\begin{document}
\begin{frontmatter}

\title{An efficient preconditioner for evolutionary partial differential equations with $\theta$-method in time discretization}

\author{Yuan-Yuan Huang\textsuperscript{\textdagger}}\ead{doubleyuans@hkbu.edu.hk}
\author{Po Yin Fung\textsuperscript{\textdagger}}\ead{pyfung@hkbu.edu.hk}
\author{Sean Y. Hon\corref{cor1}\textsuperscript{\textdagger}}\ead{seanyshon@hkbu.edu.hk}
\author{Xue-Lei Lin\textsuperscript{\textdaggerdbl}}\ead{linxuelei@hit.edu.cn}

\cortext[cor1]{Corresponding author.}

\address{\textsuperscript{\textdaggerdbl}School of Science, Harbin Institute of Technology, Shenzhen 518055, China.}
\address{\textsuperscript{\textdagger}Department of Mathematics, Hong Kong Baptist University, Kowloon Tong, Hong Kong SAR.}

\begin{abstract}
In this study, the $\theta$-method is used for discretizing a class of evolutionary partial differential equations. Then, we transform the resultant all-at-once linear system and introduce a novel one-sided preconditioner, which can be fast implemented in a parallel-in-time way. By introducing an auxiliary two-sided preconditioned system, we provide theoretical insights into the relationship between the residuals of the generalized minimal residual (GMRES) method when applied to both one-sided and two-sided preconditioned systems. Moreover, we show that the condition number of the two-sided preconditioned matrix is uniformly bounded by a constant that is independent of the matrix size, which in turn implies that the convergence behavior of the GMRES method for the one-sided preconditioned system is guaranteed. Numerical experiments confirm the efficiency and robustness of the proposed preconditioning approach.
\end{abstract}


\end{frontmatter}

\section{Introduction}
In this work, we concentrate on solving the following evolutionary partial differential equations (PDEs)
\begin{equation}\label{RSFDEs}
\begin{small}
\left\{
\begin{aligned}
&\frac{\partial u(\boldsymbol{x},t)}{\partial t} = \mathcal{L}u(\boldsymbol{x},t)+f(\boldsymbol{x},t),~~\boldsymbol{x}\in\Omega\subset \mathbb{R}^d,~~t\in(0,T],\\
&u(\boldsymbol{x},t)=0, \quad \boldsymbol{x}\in\partial\Omega,\\
&u(\boldsymbol{x},0)=\psi(\boldsymbol{x}),~\boldsymbol{x}\in\Omega,
\end{aligned}
\right.
\end{small}
\end{equation}
where $\Omega=\prod_{i=1}^d\left(\check{a}_i, \hat{a}_i\right)$ and $\partial\Omega$ denotes its boundary; $\boldsymbol{x}=\left(x_1, x_2, \ldots, x_d\right)$ is a point in $\mathbb{R}^d$. The form of equation (\ref{RSFDEs}) contains a class of equations for different choices of $\mathcal{L}$, including heat equations \cite{mcdonald2018preconditioning,lin2021all,hon2024sine}, space fractional diffusion equations \cite{lin2018efficient,donatelli2016spectral,she2023fast}, etc.

For solving linear system with Toeplitz structure arising from solving evolutionary PDEs, Krylov subspace methods are highly appropriate since the matrix-multiplications can be fast implemented via fast Fourier transforms (FFTs). However, the convergence is usually slow when the coefficient matrix is ill-conditioned. Therefore, preconditioning techniques are widely used to accelerate the convergence of the Krylov subspace methods. Here, we mention some existing preconditioners, for example,
$\tau$-preconditioners \cite{bini1990new,huang2021spectral}, circulant preconditioners \cite{lei2013circulant,lei2016multilevel}, approximate inverse preconditioners \cite{pan2014preconditioning}, structure-preserving preconditioners \cite{donatelli2016spectral}, banded preconditioners \cite{LYJ1,SLYL1}, and multigrid preconditioners \cite{moghaderi2017spectral}. In particular, $\tau$-preconditioner is getting increasingly popular in recent years since the spectra of the preconditioned matrices have been shown to be uniformly bounded by $1/2$ and $3/2$, indicating that the convergence of the preconditioned conjugate gradient method is mesh-independent \cite{huang2021spectral}. Based on this result, various fast preconditioning strategies related to $\tau$-approximation have emerged. Zhang et al. in \cite{zhang2023fast} designed a $\tau$-preconditioner to efficiently solve the discrete linear systems stemming from the fractional centered discretization of two-dimensional nonlinear Riesz space-fractional diﬀusion equations
(RSFDEs) and showed the spectra of the preconditioned matrices are independent of mesh sizes and far from zero. In \cite{huang2023tau}, Huang et al. proposed a $\tau$-preconditioner based on a fourth-order scheme and proved the uniformly bounded spectra of the preconditioned matrices for the orders $\alpha_1, \alpha_2 \in (1,1.4674)$. Later, the result was further improved by Qu et al. \cite{qu2024novel}. They adopted a novel fourth-order scheme and proposed a $\tau$-matrix based preconditioner. The uniform bounded spectrum has been proved for all $\alpha_1, \alpha_2 \in (1,2)$.  Other works related to fast solving RSFDEs, we refer readers to \cite{she2024tau,zhang2017high,zhu2021efficient,she2023fast} and the references therein. 

However, solving these conventional time-march schemes is highly time-consuming if the number of time steps is large. Therefore, all-at-once methods have been developed, which incorporate the discrete equations at all time levels into a large-scale linear system. Much work has been devoted to efficiently solving all-at-once evolutionary PDE systems. For example, when $\mathcal{L}$ is the Laplacian operator, McDonald et al. in \cite{mcdonald2018preconditioning} proposed a block circulant (BC) preconditioner to accelerate the convergence of Krylov subspace method. Lin and Ng in \cite{lin2021all} generalized the BC preconditioner by introducing a parameter $\alpha \in (0,1)$ into the top-right block of the BC preconditioner. Other related effective iterative algorithms, see the references \cite{hon2024sine,McDonald2017,liu2020fast,LiLinHonWU_2023,hondongSC2023}. 

Recently, for parallelly solving the time-space all-at-once linear system arising from the discretization of time fractional diffusion equations, efficient iterative solvers such as the generalized minimal residual (GMRES) method with the so-called two-sided preconditioning technique have been proposed in \cite{lin2021parallel,zhao2023bilateral} for non-local evolutionary equations with a weakly singular kernel, where the $\tau$-matrix is used to approximate the spatial discretization matrix such that the square root of the $\tau$-matrix approximation matrix is used as the right preconditioner, while the remainder is utilized as the left preconditioner. Theoretically, the condition numbers of the two-sided preconditioned matrices are proven to be uniformly bounded by constants independent of discretization step sizes. As a result, fast convergence of the GMRES solver for the two-sided preconditioned system can be achieved. 

Yet, the computational cost of solving the two-sided preconditioned system is notably higher compared to its one-sided counterpart \cite{lin2023single}. Developing theory ensuring the convergence of the one-sided preconditioned system is necessary. Fortunately, with the help of the theory of two-sided preconditioning technique, the effectiveness of the one-sided preconditioning technique can be demonstrated by further proving the relationship between the residuals of GMRES solver for the one-sided and two-sided preconditioned linear systems \cite{lin2023single,she2024tau}. 

Building on a similar idea, in this work, we are interested in constructing one-sided parallel-in-time preconditioners for the all-at-once system derived from a class of evolutionary PDEs. We utilize the $\theta$-scheme for the problem specified in equation \eqref{RSFDEs}, with the spatial operator $\mathcal{L}$ including (but not limited to) the following choices:
\begin{equation*}
\mathcal{L}=\left\{
\begin{aligned}
&\nabla \cdot(a(\boldsymbol{x})\nabla),~~~~~\text{variable coefficient Laplacian operator},\\
&\sum_{i=1}^d K_i \frac{\partial^{\alpha_i}}{\partial\left|x_i\right|^{\alpha_i}},~~\text{multi-dimensional Riesz fractional operator with $\alpha_i \in (1,2)$,}
\end{aligned}
\right.
\end{equation*}
and suppose the corresponding discretization matrix $\mathbf{G} \in \mathbb{R}^{N \times N}$ for $-\mathcal{L}$ is a symmetric positive definite (SPD) Toeplitz matrix. Theoretical guarantees about the convergence of GMRES solver with such one-sided preconditioners are also provided.

The remainder of this paper is organized as follows: Section \ref{Sec-2} presents the discretization of \eqref{RSFDEs} and the corresponding all-at-once linear system. Section \ref{sec:main} introduces the one-sided and the auxiliary two-sided preconditioned systems, and discusses the relationship between the residuals of the GMRES solver for both one-sided and two-sided preconditioned linear systems. Additionally, it is proven that the upper bound of the condition number for the GMRES solver of the two-sided preconditioned system remains constant, independent of the matrix size. Section \ref{sec:num} details the numerical experiments conducted to validate the efficiency and robustness of the proposed preconditioner. Conclusions are provided in Section \ref{sec:con}.

\section{Fully discretized all-at-once system}\label{Sec-2}
Let $\Delta t=T/M$ be the size of the time step with a positive integer $M$. We define the temporal partition $t_m=m\Delta t$ for $m=0,1,2, \ldots, M$. 
Let $h_i=(\hat{a}_i-\check{a}_i)/(n_i+1)$ be the size of spatial step for $i=1, \ldots, d$ with positive integer $n_i$.
For any $m, n \in \mathbb{N}$ with $m \leq n$, define the set $m \wedge n:= \{m, m+1, ..., n-1, n\}$ and denote
\begin{equation*}
N=\prod_{i=1}^d n_i, \quad n_1^{-}=n_l^{+}=1, \quad n_i^{-}=\prod_{j=1}^{i-1} n_j, i \in 2 \wedge d, \quad n_k^{+}=\prod_{j=k+1}^d n_j, k \in 1 \wedge(d-1) .
\end{equation*}
With the spatial discretization matrix $\mathbf{G} \in \mathbb{R}^{N \times N}$, we apply the $\theta$-method \cite{pan2016fast} on \eqref{RSFDEs} to obtain the following time stepping scheme
\begin{equation}\label{time_stepping}
\frac{\boldsymbol{u}^{m}-\boldsymbol{u}^{m-1}}{\Delta t} =-\theta \mathbf{G}\boldsymbol{u}^{m}-(1-\theta)\mathbf{G}\boldsymbol{u}^{m-1}+\mathbf{f}^{m,\theta},
\end{equation}
where $0 \leq \theta \leq 1$ is a weighting parameter and $\mathbf{f}^{m,\theta}=\mathbf{f}^{(m-1+\theta)\Delta t}$. Different values of $\theta$ correspond to different schemes, including the forward Euler, Crank-Nicolson and backward Euler schemes, where $\theta=0$, $\frac{1}{2}$ and 1, respectively. In the following, we focus on the values of $\theta\in[\frac{1}{2},1]$.
 
Clearly, \eqref{time_stepping} is equivalent to the following all-at-once system
\begin{align}\label{all_at_once_1}
\begin{scriptsize}
\left[\begin{array}{cccc}
\mathbf{I}_N+\theta \Delta t\mathbf{G} & & & \\
-(\mathbf{I}_N-(1-\theta)\Delta t\mathbf{G}) & \mathbf{I}_N+\theta \Delta t\mathbf{G} & & \\
& \ddots & \ddots & \\
& & -(\mathbf{I}_N-(1-\theta)\Delta t\mathbf{G}) & \mathbf{I}_N+\theta \Delta t\mathbf{G}
\end{array}\right]\left[\begin{array}{c}
\mathbf{u}^{1} \\
\mathbf{u}^{2} \\
\vdots \\
\mathbf{u}^{M}
\end{array}\right]=\left[\begin{array}{c}
(\mathbf{I}_N-(1-\theta) \Delta t\mathbf{G}) \mathbf{u}^{0}+ \mathbf{f}^{1,\theta} \\
\mathbf{f}^{2,\theta} \\
\vdots \\
\mathbf{f}^{M,\theta}
\end{array}\right].
\end{scriptsize}
\end{align}
Denote
\begin{equation*}
\mathbf{T}=\left[\begin{array}{cccc}
1 & & & \\
-1 & 1 & & \\
& \ddots & \ddots & \\
& & -1 & 1
\end{array}\right] \in \mathbb{R}^{M \times M},~\mathbf{H}_{\theta}=\Delta t\left[\begin{array}{cccc}
\theta & & & \\
1-\theta & \theta & & \\
& \ddots & \ddots & \\
& & 1-\theta & \theta
\end{array}\right] \in \mathbb{R}^{M \times M}.
\end{equation*}
Then, the coefficient matrix in \eqref{all_at_once_1} can be written in the form of Kronecker products
\begin{equation*}
\mathbf{\hat A} = \mathbf{H}_{\theta} \otimes \mathbf{G} + \mathbf{T} \otimes \mathbf{I}_N.
\end{equation*}
In order to compute the solution of (\ref{all_at_once_1}) in a parallel pattern, we exchange the order of the Kronecker product in $\mathbf{\hat A}$ such that the all-at-once system (\ref{all_at_once_1}) can be equivalently rewritten as
\begin{equation} \label{linear_sys_1}
\mathbf{\tilde Au}=\mathbf{\tilde f}
\end{equation}
with
\begin{equation*}
\mathbf{\tilde A} = \mathbf{G} \otimes \mathbf{H}_{\theta} + \mathbf{I}_{N} \otimes \mathbf{T}.
\end{equation*}

Here, by pre-multiplying $(\mathbf{I}_{N}\otimes \mathbf{H}_{\theta})^{-1}$ on both sides of \eqref{linear_sys_1}, the linear system \eqref{linear_sys_1} is equivalent to
\begin{equation}\label{linear_sys_2}
\mathbf{Au=f},
\end{equation}
where
\begin{align}\label{coef_matrix}
\mathbf{A} &= \mathbf{G} \otimes \mathbf{I}_M + \mathbf{I}_N \otimes \underbrace{\mathbf{H}_{\theta}^{-1}\mathbf{T}}_{=:\mathbf{Q}_{\theta}}
\end{align}
with
\begin{equation*}
\mathbf{Q}_{\theta}=\frac{1}{\Delta t}\left[\begin{array}{ccccc}
\frac{1}{\theta} & & & & \\
\frac{-1}{\theta^2} & \frac{1}{\theta} & & & \\
\frac{-(\theta-1)}{\theta^3} & \frac{-1}{\theta^2} & \frac{1}{\theta} & & \\
\vdots & \ddots & \ddots & \ddots & \\
\frac{-(\theta-1)^{M-2}}{\theta^n} & \cdots & \frac{-(\theta-1)}{\theta^3} & \frac{-1}{\theta^2} & \frac{1}{\theta}
\end{array}\right] \in \mathbb{R}^{M \times M},
\end{equation*}
and
\begin{equation*}
\mathbf{f} = (\mathbf{I}_N \otimes \mathbf{H}_{\theta})^{-1}\mathbf{\tilde f}.
\end{equation*}

\section{Main results}\label{sec:main}
In this section, we propose a parallel-in-time (PinT) preconditioner based on the following assumptions for the all-at-once linear system \eqref{linear_sys_2}.

\begin{assumption}\label{assumption_1}

For the SPD Toeplitz matrix $\mathbf{G}$, there exists a fast diagonalizable SPD matrix $\mathbf{\hat P}$ such that the minimum eigenvalue of $\mathbf{\hat P}$ has a lower bound independent of matrix size and the spectrum of the preconditioned matrix $\mathbf{\hat P^{-1} G}$ is uniformly bounded, i.e., 
\begin{itemize}
  \item [\bf{(i)}] $\mathbf{\hat P}$ is SPD and $\inf \limits _{N>0} \lambda_{\min }(\mathbf{\hat P}) \geq \check{c}>0$.
  \item [\bf{(ii)}] $\mathbf{\hat P}$ is fast diagonalizable. In other words, for the orthogonal diagonalization of $\mathbf{\hat P}$, $\mathbf{\hat P}=\mathbf{S}\mathbf{\Lambda}_{\mathbf{\hat P}}\mathbf{S}^{\top}$, both the orthogonal matrix $\mathbf{S}$ and its transpose $\mathbf{S}^{\top}$ have fast matrix-vector multiplications; the diagonal entries $\{\lambda_i\}_{i=1}^{N}$ of the diagonal matrix $\mathbf{\Lambda}_{\mathbf{\hat P}}=\text{diag}(\lambda_i)_{i=1}^N$ are fast computable.
  \item [\bf{(iii)}] $\sigma(\mathbf{\hat P^{-1} G}) \subset[\check{a}, \hat{a}]$ with $\check{a}$ and $\hat{a}$ being two positive constants independent of the matrix size parameter $N$.
\end{itemize}
\end{assumption}

The assumptions above can be easily satisfied. If the matrix ${\bf G}$ arises from a class of low-order discretization schemes \cite{MT2,CD2,SL1} of the multi-dimensional Riesz derivative $\mathcal{L}=\sum_{i=1}^d K_i \frac{\partial^{\alpha_i}}{\partial\left|x_i\right|^{\alpha_i}}$, then the well-known $\tau$ preconditioner \cite{huang2021spectral} is a valid candidate of $\mathbf{\hat P}$, which is both fast diagonalizable (by the multi-dimensional sine transform matrix) and SPD with its minimum eigenvalue bounded below by a constant independent of the matrix size (referring to \cite{lin2023single}); moreover, the spectra of the preconditioned matrices lie in the open interval $(1/2,3/2)$ \cite{huang2021spectral,zhang2023fast}. If ${\bf G}$ arises from the high-order discretization \cite{ding2023high} of the Riesz derivative, then the $\tau$-matrix based preconditioner proposed in \cite{qu2024novel} is a viable option of $\mathbf{\hat P}$ with fast matrix-vector multiplications and $\sigma(\mathbf{\hat P^{-1} G}) \subset (3/8,2)$. For the lower bound of minimum eigenvalue in Assumption \ref{assumption_1}$~{\bf(i)}$, see the discussion in Appendix. If the matrix ${\bf G}$ arises from the central difference discretization of the Laplacian operator $\mathcal{L}=\nabla \cdot(a(\boldsymbol{x})\nabla)$, then the constant-Laplacian matrix $\mathbf{L}_1$ (the discretization of $-\nabla^2$) is a suitable choice of $\mathbf{\hat P}$; it has been proved in \cite{lin2021parallel} that the constant-Laplacian matrix $\mathbf{L}_1$ meets Assumption \ref{assumption_1}$~{\bf(i)}$-${\bf (iii)}$.

Based on the assumptions on $\mathbf{\hat P}$, the proposed PinT preconditioner for the coefficient matrix $\mathbf{A}$ in (\ref{linear_sys_2}) is defined as
\begin{equation}\label{pre_coefficient}
\mathbf{P}_{\omega}:=\omega\mathbf{\hat P} \otimes \mathbf{I}_M + \mathbf{I}_N \otimes \mathbf{Q}_{\theta},
\end{equation}
where $\omega$ is a constant and how to choose the value of $\omega$ will be illustrated in the end of this section.
Hence, the one-sided preconditioned system is given by
\begin{equation}\label{one_side_system}
\mathbf{P}_{\omega}^{-1} \mathbf{A} \mathbf{u}=\mathbf{P}_{\omega}^{-1} \mathbf{f}.
\end{equation}
Clearly, with the fast diagonalization of $\mathbf{\hat P}$ in the Assumption \ref{assumption_1}  {\bf{(ii)}}, we know that $\mathbf{P}$ can be block diagonalized in the following form
\begin{equation*}
\begin{aligned}
\mathbf{P}_{\omega}&=\left(\mathbf{S} \otimes \mathbf{I}_N\right)\left(\omega\mathbf{\Lambda}_{\mathbf{\hat P}} \otimes \mathbf{I}_M + \mathbf{I}_N \otimes \mathbf{Q}_{\theta}\right)\left(\mathbf{S}^{\top} \otimes \mathbf{I}_N\right)\\
&:=\left(\mathbf{S} \otimes \mathbf{I}_N\right) \text{blkdiag}(\mathbf{D}_i)_{i=1}^N\left(\mathbf{S}^{\top} \otimes \mathbf{I}_N\right)
\end{aligned}
\end{equation*}
with each diagonal block $\mathbf{D}_i=\omega\lambda_i\mathbf{I}_M + \mathbf{Q}_{\theta}$ being an invertible lower triangular Toeplitz (ILTT) matrix.
Hence, the matrix-vector multiplication $\mathbf{P_{\omega}^{-1}v}$ for a given vector $\mathbf{v}$ can be implemented via the following three steps:
$$
\begin{aligned}
&\text{Step 1: Compute}~~\dot{\mathbf{v}}=\left(\mathbf{S}^{\top} \otimes \mathbf{I}_N\right) {\mathbf{v}};\\
&\text{Step 2: Compute}~~ \ddot{\mathbf{v}}=\text{blkdiag}(\mathbf{D}_i^{-1})_{i=1}^N \dot{\mathbf{v}};\\
&\text{Step 3: Compute}~~ \dddot{\mathbf{v}}=\left(\mathbf{S} \otimes \mathbf{I}_N\right) \ddot{\mathbf{v}}.
\end{aligned}
$$
From Assumption \ref{assumption_1} {\bf{(ii)}}, both Steps 1 and 3 can be fast computed. Since the inverse of an ILTT matrix is still an ILTT matrix, then the first column of each block $\mathbf{D}_i^{-1}$ can be executed within $\mathcal{O}(M\log M)$ operations utilizing the divide-and-conquer algorithm proposed in \cite{commenges1984fast}, hence Step 2 can be efficiently computed in $\mathcal{O}(MN \log M)$ operations via FFTs. Therefore, all the 3 steps above can be fast implemented.

In what follows, we will show the preconditioning effect of $\mathbf{P}_{\omega}$ for the one-sided system \eqref{one_side_system} in two parts, by introducing the following auxiliary two-sided preconditioned system
\begin{align}\label{two_side_system}
& \mathbf{P}_{l,\omega}^{-1} \mathbf{A} \underbrace{\mathbf{P}_{r,\omega}^{-1} \hat{\mathbf{u}}}_{=:\mathbf{u}}=\mathbf{P}_{l,\omega}^{-1} \mathbf{f}
\end{align}
with
\begin{equation*}
\mathbf{P}_{l,\omega}=(\omega\mathbf{\hat P})^{\frac{1}{2}} \otimes \mathbf{I}_N+(\omega\mathbf{\hat P})^{-\frac{1}{2}} \otimes \mathbf{Q}_{\theta}, \quad \mathbf{P}_{r,\omega}=(\omega\mathbf{\hat P})^{\frac{1}{2}} \otimes \mathbf{I}_N.
\end{equation*}
Firstly, we show that the convergence of the GMRES solver for the one-sided preconditioned system \eqref{one_side_system} is supported by the convergence of that for the two-sided preconditioned system \eqref{two_side_system}. Secondly, the condition number of the two-sided preconditioned matrix is to be shown to have an upper bound independent of matrix sizes. As a result of the two parts, fast convergence of GMRES for \eqref{one_side_system} can be expected.

For a square matrix $\mathbf{E} \in \mathbb{R}^{m \times m}$ and a vector $\mathbf{x} \in \mathbb{R}^{m \times 1}$, a Krylov subspace of degree $j \geq 1$ is defined as follows
\begin{equation*}
\mathcal{K}_j(\mathbf{E}, \mathbf{x}):=\operatorname{span}\left\{\mathbf{x}, \mathbf{E x}, \mathbf{E}^2 \mathbf{x}, \ldots, \mathbf{E}^{j-1} \mathbf{x}\right\}.
\end{equation*}
For a set $\mathcal{S}$ and a point $z$, we denote
\begin{equation*}
z+\mathcal{S}:=\{z+x \mid x \in \mathcal{S}\} .
\end{equation*}

\begin{lemma}\cite{Saad1}\label{lemma:gmres_residual}
For a non-singular $n \times n$ real linear system $\mathbf{My=b}$, let $\mathbf{y}^{j}$ be the iteration solution by GMRES at $j$-th $(j \geq 1)$ iteration step with $\mathbf{y}^{0}$ as initial guess. Then, the $j$-th iteration solution $\mathbf{y}^{j}$ minimizes the residual error over the Krylov subspace $\mathcal{K}_j\left(\mathbf{M}, \mathbf{r}^{0}\right)$ with $\mathbf{r}^{0}=\mathbf{b}-\mathbf{My}^{0} $, i.e.,
\begin{equation*}
\mathbf{y}^{j}=\underset{\mathbf{v} \in \mathbf{y}^{0}+\mathcal{K}_j\left(\mathbf{M}, \mathbf{r}^{0}\right)}{\arg \min }\|\mathbf{b-Mv}\|_2 .
\end{equation*}
\end{lemma}
The following theorem shows that the residual relationship between GMRES solver for one-sided preconditioned system and two-sided preconditioned system, which indicates the convergence of GMRES for one-sided preconditioned system \eqref{one_side_system} is assured by the convergence of GMRES solver for the two-sided preconditioned system \eqref{two_side_system}.

\begin{theorem}
Let $\hat{\mathbf{u}}^{0}$ be the initial guess for \eqref{two_side_system} and $\mathbf{u}^{0}=\mathbf{P}_{r,\omega}^{-1}\hat{\mathbf{u}}^{0}$ be the initial guess for \eqref{one_side_system}. Let $\mathbf{u}^{j}$ ($\hat{\mathbf{u}}^{j}$, respectively) be the $j$-th $(j\geq1)$ iteration solution derived by applying GMRES solver to \eqref{one_side_system} (\eqref{two_side_system}, respectively) with $\mathbf{u}^{0}$ ($\hat{\mathbf{u}}^{0}$, respectively) as their initial guess. Then,
\begin{equation*}
\left\|\mathbf{r}^{j}\right\|_2 \leq c\left\|\hat{\mathbf{r}}_j\right\|_2
\end{equation*}
where $\mathbf{r}^{j}:=\mathbf{P}_{\omega}^{-1} \mathbf{f}-\mathbf{P}_{\omega}^{-1} \mathbf{A} \mathbf{u}^{j}$ ($\hat{\mathbf{r}}^{j}:=\mathbf{P}_{l,\omega}^{-1} \mathbf{f}-\mathbf{P}_{l,\omega}^{-1} \mathbf{A} \mathbf{P}_{r,\omega}^{-1} \hat{\mathbf{u}}^{j}$, respectively) denotes the residual vector at the $j$-th GMRES iteration for \eqref{one_side_system} (\eqref{two_side_system}, respectively); $c=\frac{1}{\sqrt{\omega\check{c}}}$ and $\check{c}>0$ is the constant defined in Assumption \ref{assumption_1}  {\bf{(i)}} independent of both $M$ and $N$.
\end{theorem}
\begin{proof}
The direct application of Lemma \ref{lemma:gmres_residual} to the two-sided system leads to 
\begin{equation*}
\begin{aligned}
\hat{\mathbf{u}}^{j}-\hat{\mathbf{u}}^{0} &\in \mathcal{K}_j\left(\mathbf{P}_{l,\omega}^{-1} \mathbf{A} \mathbf{P}_{r,\omega}^{-1}, \hat{\mathbf{r}}^{0}\right)  ~~~~~~~ (\hat{\mathbf{r}}^{0}=\mathbf{P}_{l,\omega}^{-1} \mathbf{f}- \mathbf{P}_{l,\omega}^{-1} \mathbf{A} \mathbf{P}_{r,\omega}^{-1}\mathbf{\hat{u}}^{0})\\
& =\operatorname{span}\left\{\left(\mathbf{P}_{l,\omega}^{-1} \mathbf{A} \mathbf{P}_{r,\omega}^{-1}\right)^k\left(\mathbf{P}_{l,\omega}^{-1} \mathbf{f}- \mathbf{P}_{l,\omega}^{-1} \mathbf{A} \mathbf{P}_{r,\omega}^{-1}\mathbf{\hat{u}}^{0}\right)\right\}_{k=0}^{j-1} \\
& =\operatorname{span}\left\{\mathbf{P}_{r,\omega}\left(\mathbf{P_{\omega}^{-1} A}\right)^k \mathbf{P}_{r,\omega}^{-1}\left(\mathbf{P}_{l,\omega}^{-1} \mathbf{f}- \mathbf{P}_{l,\omega}^{-1} \mathbf{A} \mathbf{P}_{r,\omega}^{-1}\mathbf{\hat{u}}^{0}\right)\right\}_{k=0}^{j-1} \\
& =\operatorname{span}\left\{\mathbf{P}_{r,\omega}\left(\mathbf{P_{\omega}^{-1}A}\right)^k\left(\mathbf{P_{\omega}^{-1} f}- \mathbf{P_{\omega}^{-1} Au}^{0}\right)\right\}_{k=0}^{j-1}.
\end{aligned}
\end{equation*}
Then, we have
\begin{eqnarray*}
\mathbf{P}_{r,\omega}^{-1} \hat{\mathbf{u}}^{j}-\mathbf{u}^{0}&=&\mathbf{P}_{r,\omega}^{-1}\left(\hat{\mathbf{u}}^{j}-\hat{\mathbf{u}}^{0}\right) \\\nonumber
&\in& \operatorname{span}\left\{\left(\mathbf{P_{\omega}^{-1}A}\right)^k\left(\mathbf{P_{\omega}^{-1} f}- \mathbf{P_{\omega}^{-1} A}{\mathbf{u}}^{0}\right)\right\}_{k=0}^{j-1}\\\nonumber
&=&\mathcal{K}_j\left(\mathbf{P_{\omega}^{-1}A}, \mathbf{r}^{0}\right),
\end{eqnarray*}
which implies
\begin{equation*}
\mathbf{P}_{r,\omega}^{-1} \hat{\mathbf{u}}^{j} \in \mathbf{u}^{0}+\mathcal{K}_j\left(\mathbf{P_{\omega}^{-1}A}, \mathbf{r}^{0}\right) .
\end{equation*}
Again, for the one-sided system, Lemma \ref{lemma:gmres_residual} indicates that
\begin{equation*}
\mathbf{u}^{j}=\underset{\mathbf{v} \in \mathbf{u}^{0}+\mathcal{K}_j\left(\mathbf{P_{\omega}^{-1} A}, \mathbf{r}^{0}\right)}{\arg \min }\left\|\mathbf{P_{\omega}^{-1} f}- \mathbf{P_{\omega}^{-1} A}{\mathbf{v}}\right\|_2 .
\end{equation*}
Therefore,
\begin{equation*}
\begin{aligned}
\left\|\mathbf{r}^{j}\right\|_2=\left\|\mathbf{P_{\omega}^{-1} f}-\mathbf{P_{\omega}^{-1} A} \mathbf{u}^{j}\right\|_2 & \leq\left\|\mathbf{P_{\omega}^{-1} f}-\mathbf{P}_{l,\omega}^{-1} \mathbf{A} \mathbf{P}_{r,\omega}^{-1} \hat{\mathbf{u}}^{j}\right\|_2 \\
& =\left\|\mathbf{P}_{r,\omega}^{-1} \hat{\mathbf{r}}^{j}\right\|_2 \\
& =\sqrt{(\hat{\mathbf{r}}^{j})^{\top} \left((\omega\mathbf{\hat P})^{-1} \otimes \mathbf{I}_N\right) \hat{\mathbf{r}}^{j}} \\
&\leq\frac{1}{\sqrt{\omega\check{c}}}\left\|\hat{\mathbf{r}}^{j}\right\|_2  ~~~~~~~ (\mbox{by Assumption \ref{assumption_1} {\bf{(i)}}})\\
&:=c\left\|\hat{\mathbf{r}}^{j}\right\|_2.
\end{aligned}
\end{equation*}
The proof is complete.
\end{proof}
Before discussing the uniform bound of the condition number of the two-sided preconditioned matrix in \eqref{two_side_system}, we introduce some important preliminary results.

\begin{proposition}\cite{lin2018efficient}\label{main_ineq}
For positive numbers $\xi_i$ and $\zeta_i$ $(1\leq i \leq m)$, it holds that
\begin{equation*}
\min _{1 \leq i \leq m} \frac{\xi_i}{\zeta_i} \leq\left(\sum_{i=1}^m \zeta_i\right)^{-1}\left(\sum_{i=1}^m \xi_i\right) \leq \max _{1 \leq i \leq m} \frac{\xi_i}{\zeta_i}.
\end{equation*}
\end{proposition}

For any real symmetric matrices $\mathbf{C}_1, \mathbf{C}_2 \in \mathbb{R}^{k \times k}$, denote $\mathbf{C}_2 \succ$ (or $\succeq$) $\mathbf{C}_1$ if $\mathbf{C_2-C_1}$ is positive
definite (or semi-definite). Especially, we denote $\mathbf{C}_2 \succ$ (or $\succeq$) $\mathbf{O}$, if $\mathbf{C}_2$ itself is positive definite
(or semi-definite). Also, $\mathbf{C}_1 \prec$ (or $\preceq$) $\mathbf{C}_2$ and $\mathbf{O} \prec$ (or $\preceq$) $\mathbf{C}_2$  have the same meanings as those
of $\mathbf{C}_2 \succ$ (or $\succeq$) $\mathbf{C}_1$ and $\mathbf{C}_2 \succ$ (or $\succeq$) $\mathbf{O}$, respectively.

\begin{lemma}\cite{lin2021parallel}\label{ineq_inv_Toep}
Let $\mathbf{B}_1, \mathbf{B}_2 \in \mathbb{R}^{k \times k}$ be real symmetric matrices such that $\mathbf{O} \prec \mathbf{B}_1 \preceq \mathbf{B}_2$. Then, $\mathbf{O} \prec \mathbf{B}_2^{-1} \preceq \mathbf{B}_1^{-1}$.
\end{lemma}

\begin{lemma}\label{pre_ineq}
For any $\beta \in (1,2)$ and $n \in \mathbb{N}^{+}$, it holds that
\begin{equation*}
\check{a} \mathbf{I}_M \preceq \mathbf{G}^{\frac{1}{2}}\mathbf{\hat P}^{-1}\mathbf{G}^{\frac{1}{2}} \preceq \hat{a} \mathbf{I}_M
\end{equation*}
\end{lemma}
\begin{proof}
Since $\mathbf{G}^{\frac{1}{2}}\mathbf{\hat P}^{-1}\mathbf{G}^{\frac{1}{2}}$ is similar to $\mathbf{\hat P}^{-\frac{1}{2}}\mathbf{G}\mathbf{\hat P}^{-\frac{1}{2}}$, it suffices to examine the spectrum of $\mathbf{\hat P}^{-\frac{1}{2}}\mathbf{G}\mathbf{\hat P}^{-\frac{1}{2}}$.
By Assumption \ref{assumption_1} {\bf{(iii)}}, it is easy to conclude that
\begin{equation}\label{ineq_Tpeplitz}
\mathbf{O} \prec \check{a}\mathbf{\hat P} \preceq \mathbf{G} \preceq  \hat{a}\mathbf{\hat P},
\end{equation}
which indicates that
\begin{equation*}
\frac{\mathbf{z}^{\top} \mathbf{\hat P}^{-\frac{1}{2}}\mathbf{G}\mathbf{\hat P}^{-\frac{1}{2}} \mathbf{z}}{\mathbf{z}^{\top} \mathbf{z}} \xlongequal[]{\mathbf{y}=\mathbf{\hat P}^{-\frac{1}{2}}\mathbf{z}}\frac{\mathbf{y}^{\top} \mathbf{G} \mathbf{y}}{\mathbf{y}^{\top} \mathbf{\hat P} \mathbf{y}}\subset [\check{a},\hat{a}].
\end{equation*}
We then have
\begin{equation*}
\sigma\left(\mathbf{G}^{\frac{1}{2}}\mathbf{\hat P}^{-1}\mathbf{G}^{\frac{1}{2}}\right)=\sigma\left(\mathbf{\hat P}^{-\frac{1}{2}}\mathbf{G}\mathbf{\hat P}^{-\frac{1}{2}}\right) \subset [\check{a},\hat{a}],
\end{equation*}
which completes the proof.
\end{proof}

\begin{lemma}
For $\theta \in \left[\frac{1}{2},1\right]$, $\mathbf{Q}_{\theta} + \mathbf{Q}_{\theta}^{\top} \succeq \mathbf{O}$.
\end{lemma}
\begin{proof}
Let $\mathbf{B}_{\theta}=\mathbf{Q}_{\theta} + \mathbf{Q}_{\theta}^{\top}$, whose elements are explicitly given by
\begin{equation*}
\mathbf{B}_{\theta} =\frac{1}{\Delta t}\mathbf{\tilde B}_{\theta}:=\frac{1}{\Delta t}
\begin{bmatrix}
\frac{2}{\theta} & -\frac{1}{\theta^2} & -\frac{(\theta-1)}{\theta^3} & \cdots & \cdots & -\frac{(\theta-1)^{M-2}}{\theta^M} \\
-\frac{1}{\theta^2} & \frac{2}{\theta} & \ddots & \ddots & \ddots & \vdots \\
-\frac{(\theta-1)}{\theta^3} & \ddots & \ddots & \ddots & \ddots & \vdots \\
\vdots & \ddots & \ddots & \ddots & \ddots & -\frac{(\theta-1)}{\theta^3} \\
\vdots & \ddots & \ddots & \ddots & \frac{2}{\theta} & -\frac{1}{\theta^2} \\
-\frac{(\theta-1)^{M-2}}{\theta^M} & \cdots & \cdots & -\frac{(\theta-1)}{\theta^3} & -\frac{1}{\theta^2} & \frac{2}{\theta}
\end{bmatrix}.
\end{equation*}
Since $\Delta t>0$, it suffices to prove $\mathbf{\tilde B}_{\theta} \succeq \mathbf{O}$.

When $\theta = 1$, we know that
\begin{eqnarray*}
\mathbf{\tilde B}_1 =
\begin{bmatrix}
2 & -1 &  &  &  \\
-1 & 2 & \ddots &  &  \\
 & \ddots & \ddots & \ddots &  \\
 &  & \ddots & \ddots & -1 \\
 &  &  & -1 & 2 
\end{bmatrix},
\end{eqnarray*}
whose eigenvalues (i.e., $\lambda_k(\mathbf{\tilde B}_1)$) have the following closed-form:
\begin{equation*} 
\lambda_k(\mathbf{\tilde B}_1)=2+2\cos{\left(\frac{k\pi}{M+1}\right)},\quad k=1,2,...,M.
\end{equation*} Clearly, $\mathbf{\tilde B}_{1}$ is symmetric positive definite.

When $\theta=\frac{1}{2}$, we have
\begin{eqnarray*}
\mathbf{\tilde B}_\frac{1}{2} = 4
\begin{bmatrix}
b_{ij}
\end{bmatrix}_{i,j=1}^M, \quad b_{ij}=(-1)^{i+j}.
\end{eqnarray*}
which can be factorized as
\begin{eqnarray*}
    \mathbf{\tilde B}_\frac{1}{2} = 4\hat{D}\mathbb{F}_M \Lambda_{\mathbf{\tilde B}_\frac{1}{2}} \mathbb{F}_M^{*}\hat{D}^{-1}.
\end{eqnarray*}
Note that $\hat{D} = \hbox{diag}((-1)^{k})_{k=1}^M \in \mathbb{R}^{M\times M}$, $\mathbb{F}_{M} =\frac{1}{\sqrt{M}}[\theta_{M}^{(i-1)(j-1)}]_{i,j=1}^{M}\in \mathbb{C}^{M\times M}$ is the discrete Fourier matrix, and $\Lambda_{\mathbf{\tilde B}_\frac{1}{2}} = M \hbox{diag}(e_{1})$, which $e_{1}$ is the first column of the $M\times M$ identity matrix. Thus, all eigenvalues of $\mathbf{\tilde B}_{\frac{1}{2}}$ are greater than or equal to zero by its eigendecomposition as shown above. Therefore, $\mathbf{\tilde B}_{\frac{1}{2}}$ is symmetric positive semi-definite.

When $\frac{1}{2}<\theta<1$, the matrix $\mathbf{\tilde B}_{\theta}$ can be rewritten as
\begin{eqnarray*}
\mathbf{\tilde B}_{\theta} 
&=& - \frac{1}{\theta(\theta-1)}
\begin{bmatrix} 
2-2\theta  & \frac{\theta-1}{\theta} & \ldots & \cdots & \frac{(\theta-1)^{M-1}}{\theta^{M-1}}\\
 \frac{\theta-1}{\theta}  &  2-2\theta  & \ddots  & &  \vdots\\
\vdots  & \ddots & & \ddots & \frac{\theta-1}{\theta}\\
 \frac{(\theta-1)^{M-1}}{\theta^{M-1}}  & \cdots & \cdots & \frac{\theta-1}{\theta} & 2-2\theta \\
\end{bmatrix}
\\&=& \frac{1}{\theta(1-\theta)}
\left(\begin{bmatrix} 
1  & \frac{\theta-1}{\theta} & \ldots & \cdots & \frac{(\theta-1)^{M-1}}{\theta^{M-1}}\\
 \frac{\theta-1}{\theta}  &  1  & \ddots  & &  \vdots\\
\vdots  & \ddots & & \ddots & \frac{\theta-1}{\theta}\\
 \frac{(\theta-1)^{M-1}}{\theta^{M-1}}  & \cdots & \cdots & \frac{\theta-1}{\theta} & 1 \\
\end{bmatrix} 
+ (1-2\theta)\mathbf{I}_M\right)
\\&=& -\frac{(1-\rho)^2}{\rho}
\left(\begin{bmatrix} 
1  & \rho & \ldots & \cdots & \rho^{M-1}\\
\rho  &  1  & \ddots  & &  \vdots\\
\vdots  & \ddots & & \ddots & \rho\\
\rho^{M-1}  & \cdots & \cdots & \rho & 1 \\
\end{bmatrix} 
-\frac{1+\rho}{1-\rho}\mathbf{I}_M\right),~\text{where}~\rho:=\frac{\theta-1}{\theta}
\\&=&
-\frac{(1-\rho)^2}{\rho}\left(\mathbf{K}_M\left(\rho \right)-\frac{1+\rho}{1-\rho}\mathbf{I}_M \right).
\end{eqnarray*}
For $\frac{1}{2}<\theta<1$ , it gives $-1<\rho<0$. By \cite[Theorem 3.5]{fikioris2018spectral}, 
the real symmetric Toeplitz matrix $\mathbf{K}_M$ is the so-called Kac–Murdock–Szeg{\"o} matrix which is generated by the function
\begin{equation*}
    \sigma(\rho,\phi) := \frac{1-\rho^2}{1-2\rho\cos{\phi}+\rho^2}, \quad \phi\in(-\pi,\pi].
\end{equation*}
Thus, the generating function of $\mathbf{K}_M\left(\rho \right)-\frac{1+\rho}{1-\rho}\mathbf{I}_M$ is given by
\begin{eqnarray*}
    g(\rho,\phi):=\sigma(\rho,\phi)-\frac{1+\rho}{1-\rho} 
    &\geq& \frac{1+\rho}{1-\rho}-\frac{1+\rho}{1-\rho}\\
    &=& 0.
\end{eqnarray*}

Since $-\frac{(1-\rho)^2}{\rho} > 0$ and $g(\rho,\phi)$ is not identically to zeros for $-1<\rho<0$, it follows from the Grenander-Szeg{\"o} theorem \cite{chan2007introduction} that $\mathbf{K}_M\left(\rho \right)-\frac{1+\rho}{1-\rho}\mathbf{I}_M$ is SPD and so is $\mathbf{\tilde B}_{\theta}$. 

By combining the three cases above, we can conclude that $\mathbf{\tilde B}_{\theta}\succeq \mathbf{O}$ for $\frac{1}{2}\leq\theta\leq 1$. The proof is complete.
\end{proof}

Finally, we are ready to provide the main result supporting the convergence of GMRES for the two-sided preconditioned system with $\mathbf{P}_{l,\omega}^{-1} \mathbf{A} \mathbf{P}_{r,\omega}^{-1}$.

\begin{theorem}\label{main_theorem}
The condition number of the preconditioned matrix $\mathbf{P}_{l,\omega}^{-1} \mathbf{A} \mathbf{P}_{r,\omega}^{-1}$ is uniformly bounded by a constant, i.e.,
\begin{equation*}
\sup _{M, N} \kappa_2\left(\mathbf{P}_{l,\omega}^{-1} \mathbf{A} \mathbf{P}_{r,\omega}^{-1}\right) \leq \nu(\omega)
\end{equation*}
with
\begin{equation*}
\nu(\omega):=\sqrt{\frac{\hat{a} \max \left\{\frac{\hat{a}}{\omega}, 1, \frac{\omega}{\check{a}}\right\}}{\check{a}\min \left\{\frac{\check{a}}{\omega}, 1, \frac{\omega}{\hat{a}}\right\} }}~~\text{or}~~\sqrt{\frac{\hat{a} \max \left\{\frac{\hat{a}}{\omega}, \frac{\omega}{\check{a}}\right\}}{\check{a}\min \left\{\frac{\check{a}}{\omega}, \frac{\omega}{\hat{a}}\right\} }}.
\end{equation*}
In particular, $\nu(\sqrt{\hat{a}\check{a}})=\min\limits_{\omega \in(0,+\infty)} \nu(\omega)=\frac{\hat{a}}{\check{a}}$, i.e.,
\begin{equation*}
\sup _{M, N} \kappa_2\left(\mathbf{P}_{l, \omega}^{-1} \mathbf{A} \mathbf{P}_{r, \omega}^{-1}\right) \xlongequal{\omega=\sqrt{\hat{a}\check{a}}} \sup _{M, N} \kappa_2\left(\mathbf{P}_{l,\sqrt{\hat{a}\check{a}}}^{-1} \mathbf{A} \mathbf{P}_{r,\sqrt{\hat{a}\check{a}}}^{-1}\right) \leq \nu(\sqrt{\hat{a}\check{a}})=\frac{\hat{a}}{\check{a}}.
\end{equation*}
\end{theorem}
\begin{proof}
Denote $\bar{\mathbf{A}}=\mathbf{G}^{\frac{1}{2}} \otimes \mathbf{I}_N+\mathbf{G}^{-\frac{1}{2}} \otimes \mathbf{Q}_{\theta}$, then
\begin{equation*}
\begin{aligned}
& \mathbf{A}=\bar{\mathbf{A}}\left(\mathbf{G}^{\frac{1}{2}} \otimes \mathbf{I}_N\right), \\
& \left(\mathbf{P}_{l, \omega}^{-1} \mathbf{A} \mathbf{P}_{r, \omega}^{-1}\right)\left(\mathbf{P}_{l, \omega}^{-1} \mathbf{A} \mathbf{P}_{r, \omega}^{-1}\right)^{\top}=\mathbf{P}_{l, \omega}^{-1} \bar{\mathbf{A}}\left[\left(\mathbf{G}^{\frac{1}{2}}\left(\omega\mathbf{\hat P}\right)^{-1} \mathbf{G}^{\frac{1}{2}}\right) \otimes \mathbf{I}_N\right] \bar{\mathbf{A}}^{\top} \mathbf{P}_{l, \omega}^{-\top}.
\end{aligned}
\end{equation*}

Combining Lemma \ref{pre_ineq} and the Sylvester inertia law, we have
\begin{equation*}
\frac{\check{a}}{\omega} \mathbf{P}_{l, \omega}^{-1} \bar{\mathbf{A}} \bar{\mathbf{A}}^{\top} \mathbf{P}_{l, \omega}^{-\top} \preceq\left(\mathbf{P}_{l, \omega}^{-1} \mathbf{A} \mathbf{P}_{r, \omega}^{-1}\right)\left(\mathbf{P}_{l, \omega}^{-1} \mathbf{A} \mathbf{P}_{r, \omega}^{-1}\right)^{\top} \preceq \frac{\hat{a}}{\omega} \mathbf{P}_{l, \omega}^{-1} \bar{\mathbf{A}} \bar{\mathbf{A}}^{\top} \mathbf{P}_{l, \omega}^{-\top}.
\end{equation*}
In the following, we will consider the eigenvalues of the symmetric matrix $\mathbf{P}_{l, \omega}^{-1} \bar{\mathbf{A}} \bar{\mathbf{A}}^{\top} \mathbf{P}_{l, \omega}^{-\top}$.
Since
\begin{equation}\label{quadratic_quotient}
\begin{aligned}
\frac{\mathbf{z}^{\top} \mathbf{P}_{l, \omega}^{-1} \bar{\mathbf{A}} \bar{\mathbf{A}}^{\top} \mathbf{P}_{l, \omega}^{-\top} \mathbf{z}}{\mathbf{z}^{\top} \mathbf{z}}&\underbrace{=}_{\mathbf{y}=\mathbf{P}_{l, \omega}^{-\top} \mathbf{z}} \frac{\mathbf{y}^{\top} \bar{\mathbf{A}} \bar{\mathbf{A}}^{\top} \mathbf{y}}{\mathbf{y}^{\top} \mathbf{P}_{l, \omega} \mathbf{P}_{l, \omega}^{\top} \mathbf{y}}\\
&=\frac{\mathbf{y}^{\top}\left[\mathbf{G} \otimes \mathbf{I}_N+\mathbf{I}_M \otimes\left(\mathbf{Q}_{\theta}+\mathbf{Q}_{\theta}^{\top}\right)+\mathbf{G}^{-1} \otimes\left(\mathbf{Q}_{\theta} \mathbf{Q}_{\theta}^{\top}\right)\right] \mathbf{y}}{\mathbf{y}^{\top}\left[\left(\omega\mathbf{\hat P}\right) \otimes \mathbf{I}_N+\mathbf{I}_M \otimes\left(\mathbf{Q}_{\theta}+\mathbf{Q}_{\theta}^{\top}\right)+\left(\omega\mathbf{\hat P}\right)^{-1} \otimes\left(\mathbf{Q}_{\theta} \mathbf{Q}_{\theta}^{\top}\right)\right] \mathbf{y}}.
\end{aligned}
\end{equation}
It suffices to consider 
\begin{equation*}
\frac{\mathbf{y}^{\top}\left(\mathbf{G} \otimes \mathbf{I}_N\right) \mathbf{y}}{\mathbf{y}^{\top}\left(\left(\omega\mathbf{\hat P}\right) \otimes \mathbf{I}_N\right) \mathbf{y}}
\end{equation*}
and
\begin{equation*}
\frac{\mathbf{y}^{\top}\left(\mathbf{G}^{-1} \otimes\left(\mathbf{Q}_{\theta} \mathbf{Q}_{\theta}^{\top}\right)\right) \mathbf{y}}{\mathbf{y}^{\top}\left(\left(\omega\mathbf{\hat P}\right)^{-1} \otimes\left(\mathbf{Q}_{\theta} \mathbf{Q}_{\theta}^{\top}\right)\right) \mathbf{y}}.
\end{equation*}
Note that
\begin{equation*}
\mathbf{O} \prec \check{a}\mathbf{\hat P} \preceq \mathbf{G} \preceq \hat{a}\mathbf{\hat P}.
\end{equation*}
Then,
\begin{equation*}
\frac{\check{a}}{\omega} \leq \frac{\mathbf{y}^{\top}\left(\mathbf{G} \otimes \mathbf{I}_N\right) \mathbf{y}}{\mathbf{y}^{\top}\left[\left(\omega\mathbf{\hat P}\right) \otimes \mathbf{I}_N\right] \mathbf{y}} \leq \frac{\hat{a}}{\omega}.
\end{equation*}
On the other hand, combining Lemma \ref{ineq_inv_Toep} and (\ref{ineq_Tpeplitz}) gives
\begin{equation*}
\mathbf{O} \prec (\hat{a}\mathbf{\hat P})^{-1} \preceq \mathbf{G}^{-1} \preceq  (\check{a}\mathbf{\hat P})^{-1},
\end{equation*}
we then have
\begin{equation*}
\begin{aligned}
\frac{\omega}{\hat{a}}=\times \frac{\mathbf{y}^{\top}\left[\left(\hat{a}\mathbf{\hat P}\right)^{-1} \otimes\left(\mathbf{Q}_{\theta} \mathbf{Q}_{\theta}^{\top}\right)\right] \mathbf{y}}{\mathbf{y}^{\top}\left[\left(\omega\mathbf{\hat P}\right)^{-1} \otimes\left(\mathbf{Q}_{\theta} \mathbf{Q}_{\theta}^{\top}\right)\right] \mathbf{y}} & \leq \frac{\mathbf{y}^{\top}\left[\mathbf{G}^{-1} \otimes\left(\mathbf{Q}_{\theta}\mathbf{Q}_{\theta}^{\top}\right)\right] \mathbf{y}}{\mathbf{y}^{\top}\left[\left(\omega\mathbf{\hat P}\right)^{-1} \otimes\left(\mathbf{Q}_{\theta} \mathbf{Q}_{\theta}^{\top}\right)\right] \mathbf{y}} \\
& \leq  \frac{\mathbf{y}^{\top}\left[\left(\check{a}\mathbf{\hat P}\right)^{-1} \otimes\left(\mathbf{Q}_{\theta}\mathbf{Q}_{\theta}^{\top}\right)\right] \mathbf{y}}{\mathbf{y}^{\top}\left[\left(\omega\mathbf{\hat P}\right)^{-1} \otimes\left(\mathbf{Q}_{\theta} \mathbf{Q}_{\theta}^{\top}\right)\right] \mathbf{y}}= \frac{\omega}{\check{a}}.
\end{aligned}
\end{equation*}
Since all the matrices in the right hand side of (\ref{quadratic_quotient}) are SPD except for the symmetric semi-positive definite matrix $\mathbf{Q}_{\theta}+\mathbf{Q}_{\theta}^{\top}$, we know that by Proposition \ref{main_ineq}, when $\mathbf{y}^{\top} \left(\mathbf{Q}_{\theta}+\mathbf{Q}_{\theta}^{\top}\right) \mathbf{y} \neq 0$,
\begin{equation*}
\min \left\{\frac{\check{a}}{\omega}, 1, \frac{\omega}{\hat{a}}\right\} \leq \frac{\mathbf{z}^{\top} \mathbf{P}_{l, \omega}^{-1} \bar{\mathbf{A}} \bar{\mathbf{A}}^{\top} \mathbf{P}_{l, \omega}^{-\top} \mathbf{z}}{\mathbf{z}^{\top} \mathbf{z}} \leq \max \left\{\frac{\hat{a}}{\omega}, 1, \frac{\omega}{\check{a}}\right\}.
\end{equation*}
In particular, when $\theta=\frac{1}{2}$, $\mathbf{y}^{\top} \left(\mathbf{Q}_{\theta}+\mathbf{Q}_{\theta}^{\top}\right) \mathbf{y}$ may equal to zero. In this case,
\begin{equation*}
\min \left\{\frac{\check{a}}{\omega}, \frac{\omega}{\hat{a}}\right\} \leq \frac{\mathbf{z}^{\top} \mathbf{P}_{l, \omega}^{-1} \bar{\mathbf{A}} \bar{\mathbf{A}}^{\top} \mathbf{P}_{l, \omega}^{-\top} \mathbf{z}}{\mathbf{z}^{\top} \mathbf{z}} \leq \max \left\{\frac{\hat{a}}{\omega}, \frac{\omega}{\check{a}}\right\}.
\end{equation*}
Therefore,
\begin{equation*}
\frac{\check{a}}{\omega} \min \left\{\frac{\check{a}}{\omega}, 1, \frac{\omega}{\hat{a}}\right\} \mathbf{I}_{N J} \preceq\left(\mathbf{P}_{l, \omega}^{-1} \mathbf{A} \mathbf{P}_{r, \omega}^{-1}\right)\left(\mathbf{P}_{l, \omega}^{-1} \mathbf{A} \mathbf{P}_{r, \omega}^{-1}\right)^{\top} \preceq \frac{\hat{a}}{\omega} \max \left\{\frac{\hat{a}}{\omega}, 1, \frac{\omega}{\check{a}}\right\}
\end{equation*}
or
\begin{equation*}
\frac{\check{a}}{\omega} \min \left\{\frac{\check{a}}{\omega}, \frac{\omega}{\hat{a}}\right\} \mathbf{I}_{N J} \preceq\left(\mathbf{P}_{l, \omega}^{-1} \mathbf{A} \mathbf{P}_{r, \omega}^{-1}\right)\left(\mathbf{P}_{l, \omega}^{-1} \mathbf{A} \mathbf{P}_{r, \omega}^{-1}\right)^{\top} \preceq \frac{\hat{a}}{\omega} \max \left\{\frac{\hat{a}}{\omega}, \frac{\omega}{\check{a}}\right\},
\end{equation*}
and
\begin{equation*}
\kappa_2\left(\mathbf{P}_{l, \omega}^{-1} \mathbf{A} \mathbf{P}_{r, \omega}^{-1}\right)=\left\|(\mathbf{P}_{l, \omega}^{-1} \mathbf{A} \mathbf{P}_{r, \omega}^{-1})^{-1}\right\|_2 \left\|\mathbf{P}_{l, \omega}^{-1} \mathbf{A} \mathbf{P}_{r, \omega}^{-1}\right\|_2 \leq \nu(\omega).
\end{equation*}
It is easy to check that
\begin{equation*}
\nu(\sqrt{\hat{a}\check{a}})=\min _{\omega \in(0,+\infty)} \nu(\omega)=\frac{\hat{a}}{\check{a}} .
\end{equation*}
\end{proof}

\section{Numerical experiments}\label{sec:num}
In this section, we numerically examine the efficiency of the proposed preconditioner. 

In order to show the efficiency of the proposed preconditioner, numerical results with and without preconditioner are presented for comparison. In the following tables, `GMRES' represents the GMRES method without preconditioner; `PGMRES' denotes the GMRES method with the preconditioner defined in \eqref{pre_coefficient}. Also, `Error' denotes the error between the true solution and the numerical solution of the all at once system, measured in the infinity norm; `CPU(s)' displays the CPU time in seconds for solving the discrete linear system and `Iter' stands for the number of iterations of GMRES.

In all numerical tests, $\theta$ is chosen as 1/2 (i.e., the Crank-Nicolson method is used). We adopt the Matlab built-in function $\mathbf{gmres}$ with $\mathit{restart}=50$ and $\mathit{maxit}=1000$. The initial guess of GMRES methods (including PGMRES method) at each time step is chosen as the zero vector, and the stopping criterion is set as
$$
\frac{\|\mathbf{r}_k\|_2}{\|\mathbf{r}_0\|_2}\leq 10^{-8}.
$$

\begin{example}\label{example1}
\rm{
Consider the problem (\ref{RSFDEs}) with  $\mathcal{L}=\nabla \cdot(a(\boldsymbol{x})\nabla)$, $\Omega=(0,1)^2$, $T=1$ in the following two cases:
\begin{itemize}
\item Case \uppercase\expandafter{\romannumeral1}: the variable coefficient $a(\boldsymbol{x})=a(x_1, x_2)=40+x_1^{3.5}+x_2^{3.5}$, the exact solution is $u(x_1, x_2, t)=\sin(\pi x_1)\sin(\pi x_2)t^2$, and the source term is 
\begin{equation*}
\begin{aligned}
f(x_1,x_2,t)=&\sin(\pi x_1)\sin(\pi x_2)\left[2t+2\pi^2a(x_1,x_2)t^2 \right] \\
&-\pi t^2 \left[3.5x_1^{2.5}\cos(\pi x_1)\sin(\pi x_2)+ 3.5x_2^{2.5}\sin(\pi x_1)\cos(\pi x_2)\right].
\end{aligned}
\end{equation*}
\item Case \uppercase\expandafter{\romannumeral2}: the variable coefficient $a(\boldsymbol{x})=a(x_1, x_2)=(20+x_1^2)(20+x_2^2)$, the exact solution is $u(x_1, x_2, t)=e^t x_1(1-x_1)x_2(1-x_2)$, and the source term is 
\begin{equation*}
\begin{aligned}
f(x_1,x_2,t)=&e^t x_1(1-x_1)x_2(1-x_2)+2a(x_1, x_2)e^t\left(x_1(1-x_1)+x_2(1-x_2) \right) \\
&-2x_1(1-2x_1)x_2(1-x_2)(20+x_2^2)e^t-2x_2(1-2x_2)x_1(1-x_1)(20+x_1^2)e^t.
\end{aligned}
\end{equation*}
\end{itemize}
}
\end{example}

\begin{table}[!tbp]
\centering
\tabcolsep=3pt
\caption{Results of GMRES solvers for Example $1$ with $n_1+1=n_2+1=2^8$.}
\label{table1}
\begin{tabular}{cccccccc}
  \toprule
   \multirow{2}{*}{} & \multirow{2}{*}{$M$}  &   \multicolumn{3}{c}{GMRES} &  \multicolumn{3}{c}{PGMRES} \\  \cmidrule(r){3-5} \cmidrule(r){6-8} 
     &    & Error &CPU(s)  &Iter & Error &CPU(s)  &Iter  \\
\hline
  \multirow{5}{*}{Case \uppercase\expandafter{\romannumeral1}}
         &$2^{4}$    &1.8777e-3   &478.53  &3771    &1.8777e-3  &2.10 &8      \\
         &$2^{5}$    &4.5198e-4   &956.65  &3767    &4.5198e-4  &4.58 &8      \\
         &$2^{6}$    &1.0517e-4   &2144.62 &3763    &1.0517e-4  &10.68 &8      \\
         &$2^{7}$    &2.3096e-5   &4817.08 &3749    &2.3096e-5  &22.23 &8      \\
\hline
 \multirow{5}{*}{Case \uppercase\expandafter{\romannumeral2}}
         &$2^{4}$    &1.0487e-4  &552.37     &5028            &1.0484e-4  &12.29 &10      \\
         &$2^{5}$    &2.4682e-5  &1287.55    &5028            &2.4652e-5  &25.94 &10      \\
         &$2^{6}$    &5.3380e-6  &3008.99    &5028            &5.3118e-6  &52.12 &10      \\
         &$2^{7}$    &1.3342e-6  &6480.52    &4849            &1.3085e-6  &100.38&10      \\       
\bottomrule
\end{tabular}
\end{table}

\begin{table}[!tbp]
\centering
\tabcolsep=3pt
\caption{Results of GMRES solvers for Example $1$ with $M=2^{11}$.}
\label{table2}
\begin{tabular}{cccccccc}
  \toprule
   \multirow{2}{*}{} & \multirow{2}{*}{$n_i+1$}    &  \multicolumn{3}{c}{GMRES} &  \multicolumn{3}{c}{PGMRES} \\  \cmidrule(r){3-5} \cmidrule(r){6-8} 
     &    &  Error &CPU(s)  &Iter   & Error &CPU(s)  &Iter  \\
\hline
  \multirow{5}{*}{Case \uppercase\expandafter{\romannumeral1}}
         &$2^{3}$    &1.2792e-2           &67.95  &4472 &1.2792e-2   &20.79 &8      \\
         &$2^{4}$    &3.1800e-3           &237.20  &4693 &3.1800e-3   &21.42 &8      \\
         &$2^{5}$    &7.9384e-4           &1309.87  &5537 &7.9385e-4   &25.01 &8      \\
         &$2^{6}$    &1.9834e-4           &9851.06  &7816 &1.9835e-4   &48.34 &8      \\
\hline
 \multirow{5}{*}{Case \uppercase\expandafter{\romannumeral2}}
         &$2^{3}$    &3.2604e-5   &13.40     &875     &3.2603e-5   &25.10 &10      \\
         &$2^{4}$    &8.1551e-6   &65.85     &1215    &8.1536e-6   &25.77 &10      \\
         &$2^{5}$    &2.0439e-6   &446.37    &1916    &2.0422e-6   &31.13 &10      \\
         &$2^{6}$    &5.1607e-7   &4263.82   &3394    &5.1441e-7   &58.56 &10      \\      
\bottomrule
\end{tabular}
\end{table}

For Example \ref{example1}, we use the constant-Laplacian matrix to approximate the variable coefficient Laplacian matrix. Following the notation in \cite{lin2021parallel}, we denote by $\check{a}$ and $\hat{a}$ the minimum and maximum values, respectively, of the function $a(\boldsymbol{x})$. In this example, we choose $\omega=\sqrt{\hat{a}\check{a}}$ to ensure the upper bound of condition number of $\mathbf{P}_{l, \omega}^{-1} \mathbf{A} \mathbf{P}_{r, \omega}^{-1}$ achieve its minimum.

Numerical results for solving \eqref{linear_sys_2} from the two cases in Example \ref{example1}, with and without a preconditioner, are presented in Tables \ref{table1} and \ref{table2}. In this example, we utilize a second-order central difference discretization \cite{CD2} for the Riesz derivative and the preconditioner defined in \eqref{pre_coefficient} when solving \eqref{coef_matrix}. Tables \ref{table1} and \ref{table2} show that the number of iterations without the preconditioner is significantly large for large $M$ and $N$, making the GMRES method time-consuming while the proposed preconditioner substantially reduces the number of iterations required by GMRES, thereby decreasing the CPU time. Additionally, for each case, PGMRES is more accurate than GMRES and the number of iterations remains nearly constant regardless of matrix size, demonstrating that the proposed preconditioning technique is both robust and stable.
\begin{example}\label{example2}
\rm{
Consider the problem (\ref{RSFDEs}) with $\mathcal{L}=\sum_{i=1}^d K_i \frac{\partial^{\alpha_i}}{\partial\left|x_i\right|^{\alpha_i}}$, $d=2$, $T=1$, $\Omega=(0,2)^2$, $K_1=K_2=1$, $\psi(x_1,x_2)=x_1^{4}(2-x_1)^{4}x_2^{4}(2-x_2)^{4}$ and the source term
\begin{equation*}
\begin{aligned}
f(x_1,x_2,t)=&\frac{1}{2 \cos (\pi \alpha_1 / 2)} K_1 e^{-\frac{t}{3}} x_2^{4}(2-x_2)^{4} \sum_{k=0}^{4}(-1)^{4-k}2^k C_{4}^{k} \frac{\Gamma(9-k)}{\Gamma(9-k-\alpha)} g(k, \alpha_1, x_1)\\
&+\frac{1}{2 \cos (\pi \alpha_2 / 2)} K_2 e^{-\frac{t}{3}} x_1^{4}(2-x_1)^{4} \sum_{k=0}^{4}(-1)^{4-k}2^k C_{4}^{k} \frac{\Gamma(9-k)}{\Gamma(9-k-\beta)} g(k, \alpha_2, x_2) \\
&-\frac{1}{3}e^{-\frac{t}{3}} x_1^{4}(2-x_1)^{4} x_2^{4}(1-x_2)^{4},
\end{aligned}
\end{equation*}
with
\begin{equation*}
g(k, \gamma, z)=z^{8-k-\gamma}+(2-z)^{8-k-\gamma}~(k=0,1, 2,3,4,~z=x_1, x_2,~\gamma=\alpha_1, \alpha_2) .
\end{equation*}
The exact solution of this problem is
\begin{equation*}
u(x, y, t)=e^{-\frac{t}{3}} x_1^{4}(2-x_1)^{4}x_2^{4}(2-x_2)^{4}.
\end{equation*}
}
\end{example}

\begin{table}[!tbp]
\centering
\tabcolsep=3pt
\caption{Results of GMRES solvers for Example $2$ with $n_1+1=n_2+1=2^8$.}
\label{table3}
\begin{tabular}{cccccccc}
  \toprule
   \multirow{2}{*}{$(\alpha_1,\alpha_2)$} & \multirow{2}{*}{$M$}   &  \multicolumn{3}{c}{GMRES} &  \multicolumn{3}{c}{PGMRES} \\  \cmidrule(r){3-5} \cmidrule(r){6-8} 
     &    & Error &CPU(s)  &Iter &Error &CPU(s)  &Iter  \\
\hline
  \multirow{5}{*}{(1.1,1.2)}
         &$2^{3}$    &1.6496e-4   &25.15    &298       &1.6496e-4       &2.71  &21      \\   
         &$2^{4}$    &4.1112e-5   &83.81    &484       &4.1120e-5       &5.35  &21      \\  
         &$2^{5}$    &1.0271e-5   &278.10   &740       &1.0271e-5       &12.00 &21      \\  
         &$2^{6}$    &2.5661e-6   &1098.51  &1296      &2.5649e-6       &27.35 &21      \\   
\hline
 \multirow{5}{*}{(1.4,1.5)}
         &$2^{3}$    &1.6949e-4   &48.02    &593     &1.6949e-4           &2.45   &19      \\   
         &$2^{4}$    &4.2311e-5   &146.03   &900     &4.2311e-5           &4.83   &19      \\   
         &$2^{5}$    &1.0567e-5   &504.01   &1348    &1.0568e-5           &10.76  &19      \\ 
         &$2^{6}$    &2.6381e-6   &1619.21  &1934    &2.6384e-6           &24.46  &19      \\          
\hline
 \multirow{5}{*}{(1.8,1.9)}
         &$2^{3}$    &1.7674e-4   &138.24   &1642     &1.7674e-4            &1.85  &14      \\   
         &$2^{4}$    &4.4109e-5   &354.66   &2183     &4.4109e-5            &3.63  &14      \\   
         &$2^{5}$    &1.1018e-5   &1008.03  &2782     &101018e-5            &8.33  &14      \\ 
         &$2^{6}$    &2.7510e-6   &3330.02  &4080     &2.7518e-6            &18.47 &14      \\     
\hline
 \multirow{5}{*}{(1.1,1.9)}
         &$2^{3}$    &1.7200e-4   &113.04   &1292     &1.7200e-4            &2.07  &16      \\   
         &$2^{4}$    &4.2921e-5   &286.13   &1749     &4.2921e-5            &4.07  &16      \\   
         &$2^{5}$    &1.0720e-5   &903.44   &2562     &1.0721e-5            &9.18  &16      \\   
         &$2^{6}$    &2.6774e-6   &2921.51  &3508     &2.6784e-6            &20.81 &16      \\ 
\bottomrule
\end{tabular}
\end{table}

\begin{table}[!tbp]
\centering
\tabcolsep=3pt
\caption{Results of GMRES solvers for Example $2$ with $M=2^8$.}
\label{table4}
\begin{tabular}{cccccccc}
  \toprule
   \multirow{2}{*}{$(\alpha_1,\alpha_2)$} & \multirow{2}{*}{$n_i+1$}   &  \multicolumn{2}{c}{GMRES} &  \multicolumn{2}{c}{PGMRES} \\  \cmidrule(r){3-5} \cmidrule(r){6-8} 
     &    & Error &CPU(s)  &Iter &Error &CPU(s)  &Iter  \\
\hline
  \multirow{5}{*}{(1.1,1.2)}
         &$2^{3}$    &3.2563e-3   &88.70    &5872 &3.2563e-3   &5.39 &18      \\ 
         &$2^{4}$    &2.1766e-4   &185.24   &5834 &2.1766e-4   &6.56 &18      \\   
         &$2^{5}$    &1.3878e-5   &546.03   &5656 &1.3878e-5   &7.65 &19      \\   
         &$2^{6}$    &7.4802e-7   &1445.79  &5760 &7.4878e-7   &13.88&19      \\   
\hline
 \multirow{5}{*}{(1.4,1.5)}
         &$2^{3}$    &4.5421e-3   &83.35    &5498     &4.5421e-3    &5.67 &18      \\ 
         &$2^{4}$    &3.0463e-4   &174.47   &5617     &3.0463e-4    &6.52 &18      \\   
         &$2^{5}$    &1.9429e-5   &531.38   &5530     &1.9428e-5    &8.01 &17      \\   
         &$2^{6}$    &1.1407e-6   &1432.07  &5680     &1.0959e-6    &12.96&17      \\         
\hline
 \multirow{5}{*}{(1.8,1.9)}
         &$2^{3}$    &6.5320e-3   &77.74    &5175     &6.5320e-3    &5.68 &18      \\ 
         &$2^{4}$    &4.8372e-4   &166.83   &5414     &4.8372e-4    &5.82 &18      \\   
         &$2^{5}$    &3.3043e-5   &538.91   &5544     &3.3043e-5    &7.41 &16      \\   
         &$2^{6}$    &2.1066e-6   &1614.67  &6390     &2.1066e-6    &12.31 &16      \\   
\hline
 \multirow{5}{*}{(1.1,1.9)}
         &$2^{3}$    &7.3781e-3   &81.41    &5930     &7.3781e-3    &5.60 &18      \\ 
         &$2^{4}$    &5.4293e-4   &178.70   &5499     &5.4293e-4    &5.74 &18      \\   
         &$2^{5}$    &3.7326e-5   &532.12   &5584     &3.7325e-5    &7.33 &16      \\   
         &$2^{6}$    &2.3640e-6   &1566.63  &6146     &2.3640e-6    &12.33 &16      \\   
\bottomrule
\end{tabular}
\end{table}

For Example \ref{example2}, we use the fourth-order discretization scheme proposed in \cite{ding2023high} for the Riesz derivatives and the $\tau$-matrix based preconditioner proposed in \cite{qu2024novel}. Since the spectrum of the preconditioned matrix has been proved in the open interval $(3/8, 2)$ \cite{qu2024novel}, we choose $\omega=\sqrt{\hat{a}\check{a}}=\frac{\sqrt{3}}{2}$ in the numerical experiments. Errors, numbers of iterations and CPU times of GMRES method with and without preconditioner for different $\alpha_i$, $M$ and $n_i$ are displayed in Tables \ref{table3} and \ref{table4} for comparison. Similar to Example \ref{example1}, the numeral results also demonstrate satisfying performance of the PGMRES method. Specifically, it significantly reduces the number of iterations and CPU usage compared to the GMRES method. Additionally, it maintains nearly constant iteration counts for a fixed $\alpha_i$, even as matrix sizes vary.

\section{Conclusions}\label{sec:con}
In this paper, we proposed a PinT preconditioner for the all-at-once linear system arising from a class of evolutionary PDEs when the $\theta$-method is used. The proposed preconditioner was shown to be efficiently implemented. Theoretically, we explored the relationship between the residuals of the GMRES solver for one-sided and two-sided preconditioned linear systems. Additionally, it was proven that the upper bound of the condition number for the GMRES solver in the two-sided preconditioned system remained constant, independent of matrix size, thereby ensuring the convergence of the GMRES solver for the one-sided preconditioned system. The efficacy of the proposed preconditioner was well-demonstrated by the small and stable iteration counts observed in the numerical experiments.

\section*{Acknowledgments}
The work of Sean Y. Hon was supported in part by the Hong Kong RGC under grant 22300921 and a start-up grant from the Croucher Foundation. The work of Xue-lei Lin was partially supported by research grants: 12301480 from NSFC, HA45001143 from Harbin Institute of Technology, Shenzhen, HA11409084  from Shenzhen.

\bibliographystyle{plain}


\section*{Appendix}
In the appendix, we verify that when $\mathcal{L}=\sum_{i=1}^d K_i \frac{\partial^{\alpha_i}}{\partial\left|x_i\right|^{\alpha_i}}$, if $\mathbf{G}$ in (\ref{coef_matrix}) is the discretization matrix associated with the fourth-order approximation proposed in \cite{ding2023high} and $\mathbf{\hat P}$ is the corresponding $\tau$-matrix based preconditioner proposed in \cite{qu2024novel}, then $\check{c}$ in Assumption \ref{assumption_1}  {\bf{(i)}} is a constant independent of both $M$ and $N$.

\begin{proposition}
The matrices $\mathbf{\Lambda}_{\mathbf{\hat P}} \succ \mathbf{O},~\mathbf{\hat P} \succ \mathbf{O}$ and $\inf \limits _{N>0} \lambda_{\min }(\mathbf{\hat P}) \geq \check{c}>0$, where
\begin{equation}\label{lower_bound}
\check{c}:=\sum_{i=1}^d \frac{K_i}{\left(\hat{a}_i-\check{a}_i\right)^{\alpha_i}} \inf _{n_i \geq 1}(n_i+1)^{\alpha_i}\left(w_0^{\left(\alpha_i\right)}+2 \sum_{k=1}^{n_i-1} w_k^{\left(\alpha_i\right)}\right).
\end{equation}
\end{proposition}

\begin{proof}
From \cite{qu2024novel}, we know that the $\tau$-matrix based preconditioner for $\mathbf{G}$ is defined as 
\begin{equation*}
\mathbf{\hat P}:=\sum_{i=1}^d \mathbf{I}_{n_i^{-}} \otimes \tau_1\left(\mathbf{W}_{n_i}^{\left(\alpha_i\right)}\right) \otimes \mathbf{I}_{n_i^{+}}.
\end{equation*}
The symmetric Toeplitz matrix $\mathbf{W}_{n_i}^{\left(\alpha_i\right)}$ is the discretization matrix of one dimensional $\alpha_i$-th order Riesz derivative and
\begin{equation*}\label{ptau0}
\tau_1\left(\mathbf{W}_{n_i}^{\left(\alpha_i\right)}\right)=\mathbf{Q}_{n_i}^{(\alpha_i)} \tau(\tilde{\mathbf{W}}_{n_i}^{(\alpha_i)}),
\end{equation*}
where 
\begin{equation*}\label{Toeplitz-Q}
\mathbf{Q}_{n_i}^{(\alpha_i)}=\left(\begin{array}{ccccc}
1 &  & & & \\
 & 1 &  & & \\
&  & \ddots &  & \\
& &  & 1 &  \\
& & &  & 1
\end{array}\right)+\frac{\alpha_i}{24}\left(\begin{array}{ccccc}
2 & -1 & & & \\
-1 & 2 & -1 & & \\
& \ddots & \ddots & \ddots & \\
& & -1 & 2 & -1 \\
& & & -1 & 2
\end{array}\right) \in \mathbb{R}^{{n_i} \times {n_i}}
\end{equation*}
and 
$\tau(\tilde{\mathbf{W}}_{n_i}^{(\alpha_i)})$ is the $\tau$-approximation for the matrix $\tilde{\mathbf{W}}_{n_i}^{(\alpha_i)}$ derived from the fractional central difference approximation \cite{CD2} for the Riesz derivative with the first column entries of $\tilde{\mathbf{W}}_{n_i}^{(\alpha_i)}$ being $[\tilde{w}_0^{\left(\alpha_i\right)}, \tilde{w}_1^{\left(\alpha_i\right)},\cdots, \tilde{w}_{n_i}^{\left(\alpha_i\right)}]$.

Since both $\mathbf{Q}_{n_i}^{(\alpha_i)}$ and $ \tau(\tilde{\mathbf{W}}_{n_i}^{(\alpha_i)})$ are $\tau$ matrices and can be diagonalized by the discrete sine transform matrix, we know that
\begin{eqnarray*}\label{decomp_pgamma}
\tau_1(\mathbf{W}_{n_i}^{(\alpha_i)})&=&\left(\mathbf{S}_{n_i} \mathbf{\Lambda}_1 \mathbf{S}_{n_i}\right)\left(\mathbf{S}_{n_i} \mathbf{\Lambda}_2 \mathbf{S}_{n_i}\right)\\
&:=&\mathbf{S}_{n_i} \mathbf{\Lambda}_{\tau_1} \mathbf{S}_{n_i}
\end{eqnarray*}
is also a $\tau$ matrix with the diagonal matrix containing its eigenvalues 
\begin{eqnarray*}
\mathbf{\Lambda}_{\tau_1}&=&\mathbf{\Lambda}_1 \mathbf{\Lambda}_2\\
&=&\operatorname{diag}\left(\lambda_{\tau_1}^{(j_i)}\right)_{j_i=1}^{n_i},
\end{eqnarray*}
where
$\mathbf{\Lambda}_1=\operatorname{diag}\left(\lambda_1^{(j_i)}\right)_{j_i=1}^{n_i}$ and $\mathbf{\Lambda}_2=\operatorname{diag}\left(\lambda_2^{(j_i)}\right)_{j_i=1}^{n_i}$ are the eigenvalues of $\mathbf{Q}_{n_i}^{(\alpha_i)}$ and $ \tau(\tilde{\mathbf{W}}_{n_i}^{(\alpha_i)})$, respectively.
Clearly,
\begin{equation*}
\lambda_1^{(j_i)}>1,~~~ j_i \in 1 \wedge {n_i},
\end{equation*}
which implies
\begin{equation*}
\lambda_{\tau_1}^{(j_i)} > \lambda_2^{(j_i)}~~~ j_i \in 1 \wedge {n_i}.
\end{equation*}
Therefore,
\begin{equation*}
\begin{aligned}
\lambda_{\tau_1}^{(j_i)}>\lambda_2^{(j_i)} & =\frac{1}{h_i^{\alpha_i}} \tilde{w}_0^{\left(\alpha_i\right)}+\frac{2}{h_i^{\alpha_i}} \sum_{k=2}^{n_i} \tilde{w}_{k-1}^{\left(\alpha_i\right)} \cos \left(\frac{\pi j_i(k-1)}{n_i+1}\right) \\
& \geq \frac{1}{h_i^{\alpha_i}} \tilde{w}_0^{\left(\alpha_i\right)}+\frac{2}{h_i^{\alpha_i}} \sum_{k=2}^{n_i} \tilde{w}_{k-1}^{\left(\alpha_i\right)} \\
& =\frac{1}{\left(\hat{a}_i-\check{a}_i\right)^{\alpha_i}} \times\left(n_i+1\right)^{\alpha_i}\left(\tilde{w}_0^{\left(\alpha_i\right)}+2 \sum_{k=1}^{n_i-1} \tilde{w}_k^{\left(\alpha_i\right)}\right) \\
& \geq \frac{1}{\left(\hat{a}_i-\check{a}_i\right)^{\alpha_i}} \times \inf _{n \geq 1}(n+1)^{\alpha_i}\left(\tilde{w}_0^{\left(\alpha_i\right)}+2 \sum_{k=1}^{n-1} \tilde{w}_k^{\left(\alpha_i\right)}\right).
\end{aligned}
\end{equation*}
Note that $(n+1)^{\alpha_i}\left(\tilde{w}_0^{\left(\alpha_i\right)}+2 \sum_{k=1}^{n-1} \tilde{w}_k^{\left(\alpha_i\right)}\right)\gtrsim \frac{1}{\alpha_i}>0$ (see Appendix in \cite{lin2023single} for details). Thus, we conclude that
\begin{equation*}
\inf \limits _{N>0} \lambda_{\min }(\mathbf{\hat P})\geq\check{c}:= \sum_{i=1}^d \frac{K_i}{\left(\hat{a}_i-\check{a}_i\right)^{\alpha_i}} \inf _{n_i \geq 1}(n_i+1)^{\alpha_i}\left(\tilde{w}_0^{\left(\alpha_i\right)}+2 \sum_{k=1}^{n_i-1} \tilde{w}_k^{\left(\alpha_i\right)}\right)
\end{equation*}
with $\check{c}$ being a constant independent of both $M$ and $N$.
\end{proof}~

\end{document}